\documentclass[11pt]{article}
\usepackage{fullpage}
\usepackage{graphicx}
\usepackage{amsmath}
\usepackage{amsthm,amsfonts,amssymb,mathrsfs,amscd,amstext,amsbsy}
\usepackage{epic,eepic}
\usepackage[all]{xy}

\usepackage{tikz}

\newtheorem{theorem}{Theorem}[section]
\newtheorem{cor}[theorem]{Corollary}
\newtheorem{lem}{Lemma}

\newtheorem{claim}[theorem]{Claim}

\newtheorem{definition}[theorem]{Definition}
\newtheorem{obs}[theorem]{Observation}
\newtheorem{question}[theorem]{Question}

\def\qed{\hfill \ifhmode\unskip\nobreak\fi\quad\ifmmode\Box\else$\Box$\fi\\ }

\author{Xuding Zhu\thanks{Department of Mathematics, Zhejiang Normal University,  China.  E-mail: xudingzhu@gmail.com. Grant Number: CNSF11571319} }

\title{\large \bf The Z-cubes:  a hypercube variant with small diameter}

\begin{document}
\maketitle

\begin{abstract}
This paper introduces a new variant of hypercubes, which we call Z-cubes.
The $n$-dimensional Z-cube $H_n$ is obtained from two copies of the  $(n-1)$-dimensional Z-cube $H_{n-1}$
by adding a special perfect matching between the vertices of these two copies of $H_{n-1}$.
We prove that the  $n$-dimensional Z-cubes
$H_n$ has diameter $(1+o(1))n/\log_2 n$. This greatly improves on the previous known variants of
 hypercube of dimension $n$, whose diameters are all larger than   $n/3$. Moreover,   any hypercube variant
 of dimension $n$ is an
 $n$-regular graph  on $2^n$ vertices, and hence   has diameter greater than $n/\log_2 n$.
 So the Z-cubes are optimal with respect to diameters, up to an error of order $o(n/\log_2n)$.
 Another type of Z-cubes $Z_{n,k}$ which have similar structure and properties as $H_n$
 are also discussed in the last section.
\end{abstract}

 \section {Introduction}

Multiprocessor interconnection networks can be represented by graphs, where vertices represent processors and edges represent
links between processors. For the processors to communicate efficiently,
it is desired that the networks   have   small  communication delay, i.e., the graphs have small diameter.
%The interconnection networks in use usually have hierarchical structure and have  small diameter   \cite{x01}.

The hypercube network is one of the most popular interconnection networks and has been used
in both the Intel iPSC and the NCUBE/10 systems.
The $n$-dimensional hypercube $Q_n$ is a graph whose vertices are binary strings of length
$n$, with $x, y$ adjacent if and only
if $x$ and $y$ differ in exactly one bit. Alternately, we have  $Q_1=K_2$ with vertices $0$ and $1$,
and for $n \ge 2$, $Q_n$  is obtained from two copies of $Q_{n-1}$, $0Q_{n-1}$ and $1Q_{n-1}$, by
adding an edge connecting $0x$  and $1x$   for every $x \in Q_{n-1}$.
The popularity of hypercubes is due to its  simple structure,
relatively small diameter and small  vertex degree.

However, hypercubes do not have the smallest diameter for its resource.
Many variants of hypercubes
have been introduced and studied in the literature, including
 twisted cubes \cite{ap1991},  Mobius cubes \cite{CL1995}, cross cubes \cite{e1991}, etc.
 The various $n$-dimensional cubes  have
    binary strings of length $n$ as vertices, they are $n$-regular,
 and have hierarchical structure, i.e., constructed from
  lower dimensional   cubes by adding appropriate edges \cite{DE1990}.
A main feature of   the above mentioned  hypercube variants is that they have smaller diameter:
These $n$-dimensional  cubes all have diameter about half of that of the $n$-dimensional hypercube.

A natural question is whether there are $n$-dimensional cubes with smaller diameter.

In \cite{ZFJZ}, a new hypercube variant, the spined cube, was introduced. It was proved in \cite{ZFJZ} that
the $n$-dimensional spined
cube has diameter about $n/3$. Up to now, this is the only hypercube variant
which has diameter smaller than $n/2$.
Is it possible to reduce further the diameter?

In this paper, by generalizing the construction in \cite{ZFJZ},
we introduce a new variant of hypercubes $H_n$, which we call the Z-cubes.
In \cite{FH2003}, a graph $G$ is called an $n$-dimensional
 {\em bijective connection graph} (abbreviated as BC graph)   if
 $V(G)$ can be partitioned into $V_0 \cup V_1$,
 each $G[V_i]$ is an $(n-1)$-dimensional BC graph, and the edges between $V_0$ and $V_1$ is a
 perfect matching of $G$ (with $G=K_2$ when $n=1$). Many variants of hypercubes (including hypercubes itself) are special families of BC graphs.
 Our Z-cubes is also a special family of BC graphs.
For $n=1$, $H_1=K_2$. For $n \ge 2$,   $H_n$ is obtained from the disjoint union of
two copies of $H_{n-1}$ by adding a special perfect matching between vertices of these two copies of
$H_{n-1}$.
 We shall prove that
$H_n$ has diameter $ (1 + o(1)) \frac{n}{\log_2n}$.

As any $n$-dimensional hypercube variant is $n$-regular and has $2^n$ vertices,
easy counting shows that they have diameter  larger than $ n/\log_2 n$.
So the  $n$-dimensional Z-cube is  near-optimal in the sense of diameter.

This paper is organized as follows: In Section 2, we present the construction of the Z-cube $H_n$
and prove that for $n \ge 3$, $H_n$ is Hamiltonian connected. In Section 3, we prove an upper bound for the
diameter of $H_n$. In Section 4, we raise some open questions, and also present another kind of Z-cubes, $Z_{n,k}$. For each fixed integer $k$,
$Z_{n,k}$ is a family of Z-cubes, whose diameter is bounded from above by $n/(k+1) + 2^k$. By taking $k= \lceil \log_2 n - 2 \log_2 \log_2 n\rceil$,
the resulting Z-cube $Z_{n,k}$ has diameter $(1+o(1))n/\log_2 n$. These Z-cubes have simpler structure than $H_n$.
The disadvantage is that to achieve the bound $(1+o(1))n/\log_2n$, for different $n$,
we need to start with different lower dimensional cubes.

%The Z-cubes also have some other nice properties.
%We show that for $n \ge 2k+1$, the graphs $H_n$ is Hamiltonian connected.
%For any fixed integer $k$, there is a
%routing algorithm that find a relatively short path (not necessarily a shortest path)
%between any two vertices of $H_n$ in time $O(n)$.
%There is also an optimal broadcasting algorithm
%that sends a message from any source vertex to every other vertex in $n$ steps, in the model in which
%each vertex can send a message along one incident edge in each step.

 \section{The Z-cube $H_n$}

 Denote by $Z_2^n$ the set of binary strings of length $n$.
 For $x,y \in Z_2^n$, $x \oplus y$ denotes the sum of $x$ and $y$ in the group $Z_2^n$,
 i.e., $(x \oplus y)_i = x_i + y_i \pmod{2}$ (for $x \in Z_2^n$, denote by $x_i$ is the $i$th bit of $x$).

 If $x$ is a binary string of length $n_1$ and $y$ is a binary string of length $n_2$, then
 $xy$ is the {\em concatenation} of $x$ and $y$, which is a binary string of length $n_1+n_2$.
 If $Z$ is a set of binary strings, then let $xZ = \{xy: y \in Z\}$.

 For $x  \in Z_2^n$, and for $1 \le i < j \le n$, denote by $x[i,j]$ the binary string $x_ix_{i+1}\ldots x_j$.

Let $\kappa$ be the integer function defined as follows:
$$\kappa(n)= \begin{cases}
0, & \mbox{\rm if} \ $n=1$, \cr
\max\{1,  \lceil \log_2 n - 2 \log_2 \log_2 n \rceil\}, & \mbox{\rm otherwise.} \cr
\end{cases}
$$

We define a permutation $\phi$ of binary strings as follows:
\begin{definition}
Assume $x \in Z_2^n$.
Then $\phi(x) \in Z_2^n$ is the binary string such that
  \begin{eqnarray*}
 \phi(x)[1,\kappa(n)] &=& x[1,\kappa(n)] \oplus x[n-\kappa(n)+1,n],\\
 \phi(x)[\kappa(n)+1,n] &=& x[\kappa(n)+1,n].
 \end{eqnarray*}
 \end{definition}

Note that the restriction of $\phi$ to $Z_2^n$ is indeed a permutation of $Z_2^n$, with $\phi^2(x)=x$.

Now we are ready to define Z-cube $H_n$.
\begin{definition}
If $n=1$ then $H_n=K_2$, with vertices $0$ and $1$.
Assume $n \ge 2$. Then $H_n$ is obtained from two copies of $H_{n-1}$,   $0H_{n-1}$ and $1H_{n-1}$,  by
 adding edges connecting  $0x $ and $1\phi(x)$ for all $x \in H_{n-1}$.
 \end{definition}

The vertex set of $H_n$  is $Z_2^{n}$. For convenience, we  use $H_n$ to denote its vertex set as well.
Thus for $\theta \in \{0,1\}$, $\theta H_n= \theta Z_2^n=\{\theta x: x \in Z_2^n\} \subseteq Z_2^{n+1}$.

The following observation follows easily from the definition.

\begin{obs}
For $1 \le q < n$, for $a \in Z_2^q$, the subset
$aZ_2^{n-q}$ induces a copy of $H_{n-q}$.
\end{obs}

A walk $W$ in $H_n$ is viewed as
a sequence of vertices of $H_n$. If $W=(v_1,v_2, \ldots, v_m)$, then for $a \in Z_2^t$, $aW=(av_1, av_2, \ldots, av_m)$
is a walk in $H_{n+t}$ (here $av_j$ is the concatenation of two binary strings). For two walks $W_1, W_2$ in $H_n$,
if the last vertex of $W_1$ is adjacent to the first vertex of $W_2$, then the concatenation of $W_1$ and $W_2$, denoted by
$W_1 \cup W_2$ is also a walk in $H_n$.

It follows from the definition that $H_1=K_2, H_2=C_4$ and $H_3$ is as depicted in Figure 1 below.

\begin{figure}[ht]
  \begin{center}
    \includegraphics[scale=0.45]{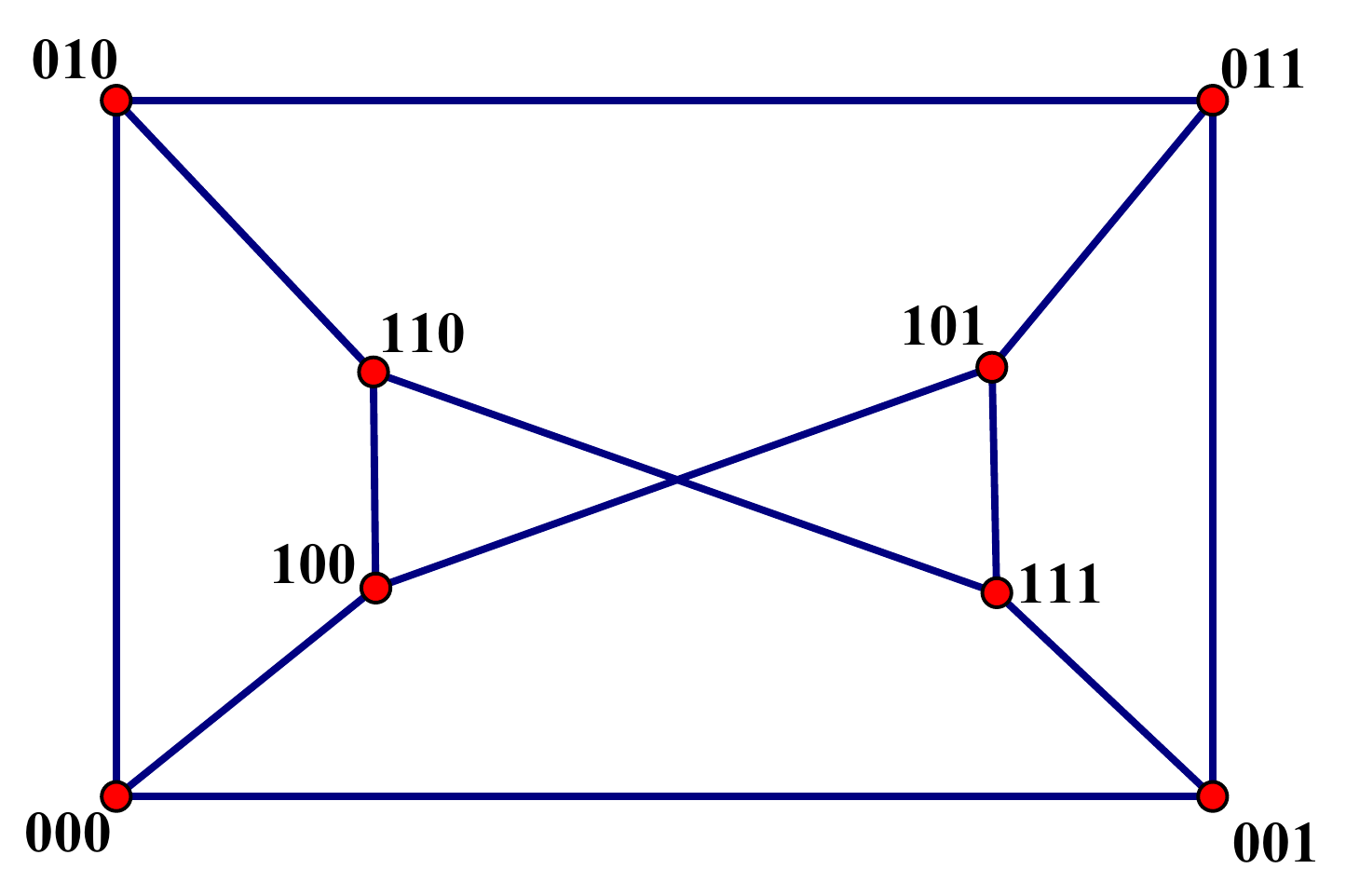}
\caption{The graph $H_3$.}
 \label{fig1}
   \end{center}
\end{figure}

A graph is called {\em Hamiltonian connected} if any two vertices are connected by a Hamiltonian path.

\begin{lem}
\label{hal}
For $n \ge 3$, the Z-cube $H_n$ is Hamiltonian connected.
\end{lem}
\begin{proof}
It is easy to check that $H_3$ is Hamiltonian connected.
From this, it is easy to derive that for any $n \ge 3$, $H_n$ is Hamiltonian connected. Indeed,
assume $n \ge 4$ and  $H_{n-1}$ is Hamiltonian connected. For $x,y \in H_n$,
if $x=0x',y=0y' \in 0H_{n-1}$ then let $P$ be a Hamiltonian path of $H_{n-1}$
connecting $x'$ and $y'$. Let $z'$ be the neighbor of $x'$ in $P$, and let $P''=P[z',y']=P-x$.
Let $P'$ be a Hamiltonian path in $H_{n-1}$ connecting $\phi(x')$ and $\phi(z')$. Then the
path start with the edge $(0x',1\phi(x'))$, followed by $1P'$, followed by the edge $(1\phi(z'), 0z')$ and then followed by $0P''$ is a Hamiltonian path of $H_n$ connecting $x$ and $y$.
If $x=0x' \in 0H_{n-1}$ and $y=1y' \in 1H_{n-1}$, then one can find a Hamiltonian path connecting $x$ and $y$ by starting from
$x$,   pass through all vertices of $0H_{n-1}$, then pass through all vertices of $1H_{n-1}$ and ending at $y$. By induction hypothesis, the paths needed in $0H_{n-1}$ and $1H_{n-1}$ exist.
\end{proof}

\section{The diameter of $H_n$}

Observe that for $x \in H_n$, a neighbor $y$ of $x$ in $H_n$ is obtained from $x$ by
changing a number of bits lying in an interval of the form $[i, i+\kappa(n-i)]$ for some $i \le n$.
As $\kappa(n-i) \le \kappa(n) $,
for $x \in H_n$,  the distance between $x$ and its complement $\bar{x}$ is at least $\lceil n/(\kappa(n)+1) \rceil$.

In the following, we give an upper bound on the diameter for $H_n$.
It follows from the definition that for any $j \ge 2$, $ \kappa(j)$ is equal to either $\kappa(j-1)$ or $\kappa(j-1)+1$.
For a positive integer $n$, let $$S(n)=\{j \le n: \kappa(j)=\kappa(j-1)+1\}$$ and let $$\sigma_n = \sum_{j \in S(n)} \frac{j}{(\kappa(j))^2}.$$

 \begin{theorem}
 \label{main}
 The graph $H_n$ defined above has diameter at most $  \frac{n}{\kappa(n)+1} +\sigma_n+2^{\kappa(n)}+ \kappa(n)$.
 \end{theorem}
 \begin{proof}
 For $n \le 3$, this is obvious. Assume $n \ge 4$.
For a positive integer $k \le n$,
a {\em $k$-robust  walk} in $H_n$ is a walk $W$ such that  for each $z \in Z_2^{k}$, there is a vertex $v \in W$ with $v[n-k+1,n]=z$.
Observe that if $k' < k$, then a $k$-robust walk in $H_n$ is also a $k'$-robust walk.
To prove Theorem \ref{main}, it suffices to prove the following claim.

 \begin{claim}
 For any $n \ge k \ge \kappa(n)$, for any $x,y \in H_n$ (not necessarily distinct),
 there is a   $k$-robust  walk from $x$ to $y$ of length at most $  n/(\kappa(n)+1) +\sigma_n + 2^{k}+k$.
 \end{claim}
 \begin{proof}
 We prove this claim by induction on $n$.
 For $k \le n \le 2k$, we claim that there is a $k$-robust walk of length $n-k+2^k$ connecting $x$ and $y$.
If $n=k$, then let $W$ be a Hamiltonian path in $H_n$ connecting $x$ and $y$, which has length at most $2^k$.
Assume $n \ge k+1$ and the claim is true for $n-1$. If $x,y$ are in the same copy of $H_{n-1}$, then the conclusion follows from
the induction hypothesis.   Assume $x = 0x' \in 0H_{n-1}$ and $y = 1y'\in 1H_{n-1}$.
By induction hypothesis, there is a $k$-robust walk $W$ of length $n-1-k+2^k$ in $H_{n-1}$ connecting $\phi(x')$ and $y'$.
Then the walk start from the edge $(0x', 1 \phi(x'))$ followed by $1W$ is a $k$-robust walk of length at most
$n-k+2^k$ in $H_n$ connecting $x$ and $y$.

 Assume $n \ge 2k+1$ and the claim above holds for $ n' < n$.
 By definition, $H_n$ is obtained from two copies of $H_{n-1}$, $0H_{n-1}$ and $1H_{n-1}$,
 by adding an edge  $0x \sim 1\phi(x)$ for each $x \in H_{n-1}$.
 If $x,y$ belong to the same copy of $H_{n-1}$, then we are done by induction hypothesis,
 Assume   $x   \in 0Q_{n-1}$ and $y  \in 1Q_{n-1}$.

 Let $a=x[2,\kappa(n)+1]$, $x'=x[\kappa(n)+2,n]$,    $b=y[2,\kappa(n)+1]$ and $y'=y[\kappa(n)+2,n]$.
 So  $x = 0ax', y=1by'$ and $x', y' \in  H_{n-\kappa-1}$.

 By induction  hypothesis, there is
 a $k$-robust  walk $W$ in $H_{n-\kappa(n)-1}$ from
 $x'$ to $y'$ of length at most $$   \frac{n-\kappa(n)-1}{ \kappa(n-\kappa(n)-1)+1} +\sigma_{n-\kappa(n)-1}+2^k+k.$$

 Let $c =a \oplus b$. Since $W$ is a $k$-robust walk, there is a vertex $v \in W$
 with $v[n-\kappa(n)+1,n]=c$. Let $W_1=W[x',v]$ be the subwalk of $W$ from $x'$ to $v$,
 and $W_2=W[v, y']$ be the subwalk of $W$ from $v$ to $y'$. Then $0aW_1$   is a walk in $H_n$
 from $x=0ax'$ to $0av$, and $1bW_2$ is a walk in $H_n$ from $1bv$ and $1by'=y$. Since
 $$\phi(av)[1, \kappa(n)]=  a \oplus v[n-\kappa(n)+1,n] = a \oplus c = b$$ and $$\phi(av)[\kappa+1,n]=v,$$ we have $\phi(av)=bv$ and hence $0av $
 is adjacent to $ 1bv$ in $H_n$. Now the concatenation $0aW_1\cup 1bW_2$ of $0aW_1$ and
 $1bW_2$ is a $k$-robust  walk in $H_n$ connecting $x$ and $y$, whose  length is
 $$|W|+1 \le   \frac{n-\kappa(n)-1}{ \kappa(n-\kappa(n)-1)+1} +\sigma_{n-\kappa(n)-1}+2^k+k +1.$$

 It remains to show that
 $$\frac{n-\kappa(n)-1}{ \kappa(n-\kappa(n)-1)+1} +\sigma_{n-\kappa(n)-1}+2^k+k +1 \le \frac{n}{ \kappa(n) +1}  +\sigma_{n}+2^k+k.$$
 I.e.,
 $$\frac{n-\kappa(n)-1}{ \kappa(n-\kappa(n)-1)+1} +\sigma_{n-\kappa(n)-1}  +1 \le \frac{n}{ \kappa(n) +1}  +\sigma_{n} .$$
 It follows easily from the definition that    $\kappa(n-\kappa(n)-1 )$ is either equal to $\kappa(n)$ or equal to $\kappa(n)-1$.
 In the former case, we have $\sigma_n = \sigma_{n-\kappa(n)-1}$, and hence
 $$   \frac{n-\kappa(n)-1}{ \kappa(n-\kappa(n)-1)+1} +\sigma_{n-\kappa(n)-1}+1 = \frac{n}{ \kappa(n)+1} +\sigma_{n}. $$
 In the later case, $$\sigma_n = \sigma_{n-\kappa(n)-1} + \frac{j}{(\kappa(j))^2}$$
 for some $n-\kappa(n) \le j \le n$, and hence
$$\sigma_n \ge \sigma_{n-\kappa(n)-1} + \frac{n-\kappa(n)}{\kappa(n)^2}.$$
Therefore,
\begin{eqnarray*}
  \frac{n-\kappa(n)-1}{ \kappa(n-\kappa(n)-1)+1} +\sigma_{n-\kappa(n)-1} +1&\le &  \frac{n-\kappa(n)-1}{ \kappa(n) } +\sigma_{n}- \frac{n-\kappa(n)}{\kappa(n)^2}+1 \\
  &=&  \frac{n-\kappa(n)-1}{ \kappa(n) +1} + \frac{n-\kappa(n)-1} {\kappa(n)(\kappa(n)+1)}+\sigma_{n}- \frac{n-\kappa(n)}{\kappa(n)^2}+1 \\
  &\le& \frac{n}{ \kappa(n) +1}  +\sigma_{n}.
  \end{eqnarray*}

 \end{proof}

This  completes the proof of Theorem \ref{main}.
\end{proof}

\begin{cor}
\label{cor}
For any positive integer $n$, the Z-cube $H_n$ has diameter at most $  (1+o(1))\frac{n}{\log_2 n}$.
\end{cor}
\begin{proof}
By Theorem \ref{main}, $H_n$ has diameter at most
\begin{eqnarray*}
& & \frac{n}{\kappa(n) +1} + \sigma_n + 2^{\kappa(n)} + \kappa(n) \\
& \le & \frac{n}{\log_2n} \frac{\log_2 n}{\log_2 n - 2\log_2 \log_2 n+1} + \sigma_n+2^{\kappa(n)} + \kappa(n)\\
&=& (1+o(1)) \frac{n}{\log_2 n} + \sigma_n+2^{\kappa(n)} + \kappa(n).
 \end{eqnarray*}

As $\kappa(n) \le \log_2n-2 \log_2 \log_2 n +1$, we have
\begin{eqnarray*}
2^{\kappa(n)}+\kappa(n)
 &\le& 2^{\kappa(n)+1}  \\
 &\le&  \frac{4n}{(\log_2n)^2}\\
 &=& o(n/\log_2n).
 \end{eqnarray*}
To prove that $H_n$ has diameter at most $(1+o(1))n/\log_2n$, it remains to show that $\sigma_n = o(n/ \log_2n)$.
By definition,
$$\sigma_n = \sum_{j \in S(n)} \frac{j}{(\kappa(j))^2} = \sum_{i=1}^{\kappa(n)}\frac{\kappa^{-1}(i)}{i^2}, \eqno(1)$$
here we let $\kappa^{-1}(i) = \min\{j: \kappa(j)=i\}.$
As $\log_2i -2 \log_2 \log_2 i \le \kappa(i)$,
we have $$i \le (\log_2 i)^2 2^{\kappa(i)}. \eqno(2)$$
There is a constant $c$ such that for $i \ge c$,
  $\frac{1}{2} \log_2 i \ge 2 \log \log_2 i$. Plug this into the definition of $\kappa$, we have
$$\log_2 i \le 2 \kappa(i). \eqno(3)$$
Plug (3) into (2), we have
$$i \le (2 \kappa(i))^2 2^{\kappa(i)}. \eqno(4)$$
So for $i \ge c$, $$\kappa^{-1}(i) \le 4 i^22^i. \eqno(5)$$
Plug (5) into (1), we have
\begin{eqnarray*}
 \sigma_n \le 4 \sum_{i=c+1}^{\kappa(n)} 2^i + C,
\end{eqnarray*}
where $C = \sum_{i=}^c \frac{\kappa^{-1}(i)}{i^2}$ is a constant. Hence
$$\sigma_n \le 2^{\kappa(n)+3}+C \le 16\frac{n}{(\log_2n)^2} + C = o(n/\log_2 n).$$
This completes the proof of the corollary.
\end{proof}

\section{Some questions and discussions}

The upper bound on the diameter of $H_n$ given in Theorem \ref{main} is probably not tight. One natural problem is to determine the diameter
of $H_n$.

\begin{question}
What is  the diameter of $H_n$?
\end{question}

One nice property of the hypercube is that one can easily find a shortest path between any two vertices.
However, it seems not easy to find the shortest path between two vertices in $H_n$.

\begin{question}
Is there an algorithm that finds a shortest path between any two vertices of $H_n$ in time polynomial of $n$?
\end{question}

Since the diameter of $H_n$ is much smaller than that of the hypercube $Q_n$, it is perhaps good enough
to have a quick algorithm that finds a short path (not necessarily the shortest) path between two vertices.
For this purpose, it is also desirable to have knowledge on the average distance between two vertices of $H_n$.

\begin{question}
What is the average distance between two vertices of $Q_n$ ?
\end{question}

We have shown that the Z-cube $H_n$ is Hamiltonian connected for $n \ge 3$. Fault tolerance has been studied a lot for
various variants of hypercubes \cite{HTHH2002}. The same question is also interesting for Z-cubes.

\begin{question}
What is the minimum number of vertices and/or edges whose deletion results in
a non-Hamiltonian graph?
\end{question}

Not many automorphisms of $H_n$ are known. It is unknown if $H_n$ is vertex-transitive.

\begin{question}
What is the automorphism group of $H_n$?
\end{question}

One can define a variant of hypercube which is similar to $H_n$, but have simpler structure.
The precise definition is as follows:

For a fixed non-negative integer $k$,   we define a permutation $\phi_k$ of binary strings as follows:
\begin{definition}
Assume $x \in Z_2^n$. Then $\phi_k(x) \in Z_2^n$ is the binary string such that
\begin{itemize}
\item If $ n \le k$, then $\phi_k(x)=x$.
\item If $k+1 \le n \le 2k$, then
 \begin{eqnarray*}
 \phi_k(x)[1,n-k]&=&x[1,n-k] \oplus x[k+1,n], \\
 \phi_k(x)[n-k+1,n]&=&x[n-k+1,n].
 \end{eqnarray*}
\item If $n \ge 2k+1$, then
  \begin{eqnarray*}
 \phi_k(x)[1,k] &=& x[1,k] \oplus x[n-k+1,n],\\
 \phi_k(x)[k+1,n] &=& x[k+1,n].
 \end{eqnarray*}
 \end{itemize}
 \end{definition}

We define a family of Z-cubes $Z_{n,k}$ as follows:
\begin{definition}
If $n=1$ then $Z_{n,k}=K_2$, with vertices $0$ and $1$.
Assume $n \ge 2$. Then $Z_{n,k}$ is obtained from two copies of $Z_{n-1,k}$,   $0Z_{n-1,k}$ and $1Z_{n-1,k}$,  by
 adding edges connecting  $0x $ and $1\phi_k(x)$ for all $x \in Z_{n-1,k}$.
 \end{definition}

 A similar (but simpler) argument as the proof of Theorem \ref{main} shows that
  the graph $Z_{n,k}$  has diameter at most $  n/(k+1)  + 2^{k}$.

For each $n$, let $Z^*_n=Z_{n, \kappa(n)}$. Then it is easy to verify (with an argument simpler than the proof of
Corollary \ref{cor}) that $Z^*_n$ has diameter at most
\begin{eqnarray*}
 \left(1+ \frac{2\log_2 \log_2 n +1}{\log_2 n- 2 \log_2 \log_2 n -1 }+ \frac{1}{\log_2n}\right)  \frac{n}{\log_2 n} 
  =\left(1+o(1)\right) \frac{n}{\log_2 n}.
\end{eqnarray*}

Although both Z-cubes $H_n$ and $Z^*_n$ have diameter $(1 +o(1))n/\log_2 n$, the upper bound for
the diameter of $Z^*_n$ obtained above is smaller than the upper bound for the the diameter of $H_n$   in  Theorem \ref{main}.
So we may prefer the Z-cube $Z^*_n$ to $H_n$, if we aim to construct a cube of a fixed dimension.

The disadvantage of $Z^*_n$ is that for each $n$, to construct a Z-cube $Z^*_n$, we start with
a special $k$ from the very beginning. If we want to increase the dimension of the cube later, we may need to start the
construction from a larger $k$. So we cannot just expand the existing cube. Instead we need to
start the construction process from the very beginning. Nevertheless,
if we fix an integer $k$ at the beginning, as $n$ goes to infinity, the diameter of $Z_{n,k}$ is at most $n/(k+1)+2^k$,
which is probably good enough for many purposes. At least it is better than the earlier variants of hypercubes, for the sake of diameter.

Also, it seems that   the Z-cubes $Z_{n,k}$   have simpler structure than $H_n$,
and it is probably easier to design a routing algorithm in $Z_{n,k}$.
Nevertheless, the   questions asked above are also open for the Z-cubes $Z_{n,k}$:
(1) What is the diameter of $Z_{n,k}$? (2) Is there an algorithm that finds a shortest path between any two vertices of $Z_{n,k}$ in time
polynomial in $n$?  (3) What is the average distance between two vertices of $Z_{n,k}$?
(4) What is the minimum number of vertices and/or edges whose deletion results in
a non-Hamiltonian graph? (5)
What is the automorphism group of $H_n$? The Z-cubes are new variants of hypercubes. Many problems about
hypercubes and their variants should be interesting for Z-cubes as well.

  \section*{Acknowledgment}

The author would like to thank Professor Meijie Ma for bringing reference \cite{ZFJZ} to his attention.

\end{document}